\documentclass{amsart}
\usepackage{mathrsfs,latexsym,amsfonts,amssymb}

\newtheorem{theorem}{Theorem}[section]
\newtheorem{lemma}[theorem]{Lemma}
\newtheorem{corollary}[theorem]{Corollary}

\theoremstyle{definition}
\newtheorem{definition}[theorem]{Definition}
\newtheorem{proposition}[theorem]{Proposition}
\theoremstyle{remark}

\begin{document}

\title[A note on free paratopological groups]
{A note on free paratopological groups}

\author{Fucai Lin}
\address{Fucai Lin: Department of Mathematics and Information Science,
Zhangzhou Normal University, Zhangzhou 363000, P. R. China}
\email{linfucai2008@yahoo.com.cn}

\thanks{Supported by the NSFC (No. 10971185, No. 11201414) and the Natural Science Foundation of Fujian Province (No. 2012J05013) of China.}


\begin{abstract}
In this paper, we mainly discuss some generalized metric properties and the character of the free paratopological groups, and extend several results valid
for free topological groups to free  paratopological groups.
\end{abstract}

\maketitle

\section{Introduction}
In 1941, free topological groups were introduced by A.A. Markov in \cite{MA} with the clear idea of extending the well-known construction of a free group from group theory to topological groups. Now, free topological groups have become a powerful tool of study in the theory of topological groups and serve as a source of various examples and as an instrument for proving new theorems, see \cite{A2008, GM, LF, NP}.

It is well known that paratopological groups are
good generalizations of topological groups, see e.g. \cite{A2008}. The Sorgenfrey line
(\cite[Example 1.2.2]{E1989}) with the usual addition is a
first-countable paratopological group but not a topological group. The absence of continuity of inversion,
the typical situation in paratopological groups, makes the study in this area
very different from that in topological groups. Paratopological groups attract
a growing attention of many mathematicians and articles in recent years, see \cite{AR2005, CJ, LF1, LF2, LF, LC, LL2010}.
As in free topological groups, S. Romaguera, M. Sanchis and M.G. Tkackenko in \cite{RS} define free paratopological groups. Recently, N.M. Pyrch has investigated some properties of free paratopological groups, see \cite{PN1, PN, PN2}. In this paper, we will discuss some generalized metric properties and the character of the free paratopological groups, and extend several results valid
for free topological groups to free  paratopological groups.

In section 3, we show,
for example, that the groups $FP(X)$ and $AP(X)$ are $P$-spaces if and only if $X$ is a $P$-space, that if $X$ is a $P$-space then the groups $FP(X)$ and $AP(X)$ have the direct limit property, that if $Y$ is a closed subspace of a Tychonoff space $X$ then the subgroup $PG(Y, X)$ of $PG(X)$ generated by $Y$ is closed in $PG(X)$. In section 4, we mainly discuss the character of free abelian paratopological groups, and extend several results valid
for free abelian topological groups to free abelian paratopological groups.
\maketitle

\section{Preliminaries}
All spaces are $T_0$ unless stated otherwise. We denote by $\mathbb{N}$ the set of all natural
numbers. The letter $e$
denotes the neutral element of a group. Readers may consult
\cite{A2008, E1989, Gr1984} for notations and terminology not
explicitly given here.

Firstly, we introduce some notions and terminology.

Recall that a {\it topological group} $G$ is a group $G$ with a
(Hausdorff) topology such that the product mapping of $G \times G$ into
$G$ is jointly continuous and the inverse mapping of $G$ onto itself
associating $x^{-1}$ with an arbitrary $x\in G$ is continuous. A {\it
paratopological group} $G$ is a group $G$ with a topology such that
the product mapping of $G \times G$ into $G$ is jointly continuous.

\begin{definition}\cite{MA}
Let $X$ be a subspace of a topological group $G$. Assume that
\begin{enumerate}
\item The set $X$ generates $G$ algebraically, that is $<X>=G$;

\item  Each continuous mapping $f: X\rightarrow H$ to a topological group $H$ extends to a continuous homomorphism $\hat{f}: G\rightarrow H$.
\end{enumerate}
Then $G$ is called the {\it Markov free topological group on} $X$ and is denoted by $F(X)$.
\end{definition}

\begin{definition}\cite{RS}
Let $X$ be a subspace of a paratopological group $G$. Assume that
\begin{enumerate}
\item The set $X$ generates $G$ algebraically, that is $<X>=G$;

\item  Each continuous mapping $f: X\rightarrow H$ to a paratopological group $H$ extends to a continuous homomorphism $\hat{f}: G\rightarrow H$.
\end{enumerate}
Then $G$ is called the {\it Markov free paratopological group on} $X$ and is denoted by $FP(X)$.
\end{definition}

Again, if all the groups in the above definitions are Abelian, then we get the definitions of the {\it Markov free Abelian topological group} and the {\it Markov free Abelian paratopological group on} $X$ which will be denoted by $A(X)$ and $AP(X)$ respectively.

By a {\it quasi-uniform space} $(X, \mathscr{U})$ we mean the natural analog of a {\it uniform space} obtained by dropping the symmetry axiom. For each quasi-uniformity $\mathscr{U}$ the filter  $\mathscr{U}^{-1}$ consisting of the inverse relations $U^{-1}=\{(y, x): (x, y)\in U\}$ where $U\in\mathscr{U}$ is called the {\it conjugate quasi-uniformity} of $\mathscr{U}$.

We also recall that the {\it universal quasi-uniformity} $\mathscr{U}_{X}$ of a space $X$ is the finest quasi-uniformity on $X$ that induces on $X$ its original topology. Throughout this paper, if $\mathscr{U}$ is a quasi-uniformity of a space $X$ then $\mathscr{U}^{\ast}$ denotes the
smallest uniformity on $X$ that contains $\mathscr{U}$, and $\tau (\mathscr{U})$ denotes the topology of $X$ generated by $\mathscr{U}$. A quasi-uniform space $(X, \mathscr{U})$ is called {\it bicomplete} if $(X, \mathscr{U}^{\ast})$ is complete.

\begin{definition}
A {\it quasi-pseudometric} $d$ on a set $X$ is a function
from $X\times X$ into the set of non-negative real numbers such that for
$x, y, z\in X$: (a) $d(x, x)=0$ and (b) $d(x, y)\leq d(x, z)+d(z, y)$. If $d$ satisfies the additional condition (c) $d(x, y)=0\Leftrightarrow x=y$, then $d$ is called a {\it quasi-metric} on $X$.
\end{definition}

Every quasi-pseudometric $d$ on $X$ generates a topology $\mathscr{F}(d)$ on $X$ which has as a
base the family of $d$-balls $\{B_{d}(x, r): x\in X, r>0\}$, where $B_{d}(x, r)=\{y\in X:
d(x, y)< r\}$.

A topological space $(X, \mathscr{F})$ is called {\it quasi-(pseudo)metrizable} if there is a quasi-
(pseudo)metric $d$ on $X$ compatible with $\mathscr{F}$, where $d$ is compatible with $\mathscr{F}$ provided $\mathscr{F}=\mathscr{F}(d).$

Denote by $\mathscr{U}^{\star}$ the upper quasi-uniformity on $\mathbb{R}$ the standard base of which consists of the sets $$U_{r}=\{(x, y)\in \mathbb{R}\times \mathbb{R}: y<x+r\},$$where $r$ is an arbitrary positive real number.

\begin{definition}
Given a group $G$ with the neutral element $e$, a function $N: G\rightarrow [0,\infty)$
is called a {\it quasi-prenorm} on $G$ if the following conditions are satisfied:
\begin{enumerate}
\item $N(e)=0$; and
\item $N(gh)\leq N(g)+N(h)$ for all $g, h\in G$.
\end{enumerate}
\end{definition}

Let $X$ be a space and $p\in X$. The character for $p$,  character for $X$ and weight for $X$ are defined as follows respectively:
$$\chi(p, X)=\mbox{min}\{|\mathscr{V}|: \mathscr{V}\ \mbox{is a local base for}\ p\};$$
$$\chi(X)=\mbox{sup}\{\chi(p, X): p\in X\}+\omega;$$
$$\omega(X)=\mbox{min}\{|\mathscr{B}|: \mathscr{B}\ \mbox{is a  base for}\ X\}.$$

Throughout this paper, we use $G(X)$ to denote the topological groups $F(X)$ or $A(X)$, and $PG(X)$ to denote the paratopological groups $FP(X)$ or $AP(X)$. For a subset $Y$ of a space $X$, we use $G(Y, X)$ and $PG(Y, X)$ to denote the subgroups of $G(X)$ and $PG(X)$ generated by $Y$ respectively. Moreover, we denote the abstract groups of $F(X), FP(X)$ by $F_{a}(X)$ and of $A(X)$ and $AP(X)$ by $A_{a}(X)$, respectively.

Since $X$ generates the free group $F_{a}(X)$, each element $g\in F_{a}(X)$ has the form $g=x_{1}^{\varepsilon_{1}}\cdots x_{n}^{\varepsilon_{n}}$, where $x_{1}, \cdots, x_{n}\in X$ and $\varepsilon_{1}, \cdots, \varepsilon_{n}=\pm 1$. This word for $g$ is called {\it reduced} if it contains no pair of consecutive symbols of the form $xx^{-1}$ or $x^{-1}x$. It follow that if the word $g$ is reduced and non-empty, then it is different from the neutral element of $F_{a}(X)$. In particular, each element $g\in F_{a}(X)$ distinct from the neutral element can be uniquely written in the form $g=x_{1}^{r_{1}}x_{2}^{r_{2}}\cdots x_{n}^{r_{n}}$, where $n\geq 1$, $r_{i}\in \mathbb{Z}\setminus\{0\}$, $x_{i}\in X$, and $x_{i}\neq x_{i+1}$ for each $i=1, \cdots, n-1$. Such a word is called the {\it normal form} of $g$. Similar assertions are valid for $A_{a}(X)$.

We now outline some of the ideas of \cite{RS} in a form suitable for our applications.

Suppose that $e$ is the neutral element of the abstract free group $F_{a}(X)$ on $X$, and suppose that $\rho$ is a fixed quasi-pseduometric on $X$ which is bounded by 1. Extend $\rho$ from $X$ to a quasi-pseudometric $\rho_{e}$ on $X\cup\{e\}$ by putting
\[\rho_{e}(x, y)=\left\{
\begin{array}{lll}
0, & \mbox{if } x=y,\\
\rho(x, y), & \mbox{if } x, y\in X,\\
1, & \mbox{otherwise}\end{array}\right.\]
for arbitrary $x, y\in X\cup\{e\}$. By \cite{RS}, we extend $\rho_{e}$ to a quasi-pseudometric $\rho^{\ast}$ on $\tilde{X}=X\cup\{e\}\cup X^{-1}$ defined by
\[\rho^{\ast}(x, y)=\left\{
\begin{array}{lll}
0, & \mbox{if } x=y,\\
\rho_{e}(x, y), & \mbox{if } x, y\in X\cup\{e\},\\
\rho_{e}(y^{-1}, x^{-1}), & \mbox{if } x, y\in X^{-1}\cup\{e\},\\
2, & \mbox{otherwise}\end{array}\right.\] for arbitrary $x, y\in\tilde{X}$.

Let $A$ be a subset of $\mathbb{N}$ such that $|A|=2n$ for some $n\geq 1$. A {\it scheme} on $A$ is a partition of $A$ to pairs $\{a_{i}, b_{i}\}$ with $a_{i}<b_{i}$ such that each two intervals $[a_{i}, b_{i}]$ and $[a_{j}, b_{j}]$ in $\mathbb{N}$ are either disjoint or one contains the other.

If $\mathscr{X}$ is a word in the alphabet $\tilde{X}$, then we denote the reduced form and the length of  $\mathscr{X}$ by $[\mathscr{X}]$ and $\ell (\tilde{X})$ respectively.

For each $n\in \mathbb{N}$, let $\mathscr{S}_{n}$ be the family of all schemes $\varphi$ on $\{1, 2, \cdots, 2n\}$. As in \cite{RS}, define
$$\Gamma_{\rho}(\mathscr{X}, \varphi)=\frac{1}{2}\sum_{i=1}^{2n}\rho^{\ast}(x_{i}^{-1}, x_{\varphi (i)}).$$
Then we define a quasi-prenorm $N_{\rho}: F_{a}(X)\rightarrow [0, +\infty)$ by setting $N_{\rho}(g)=0$ if $g=e$ and $$N_{\rho}(g)=\inf\{\Gamma_{\rho}(\mathscr{X}, \varphi): [\mathscr{X}]=g, \ell (\tilde{X})=2n, \varphi\in\mathscr{S}_{n}, n\in \mathbb{N}\}$$ if $g\in F_{a}(X)\setminus\{e\}$. It follows from Claim 3 in \cite{RS} that $N_{\rho}$ is an invariant quasi-prenorm on $F_{a}(X)$. Put $\hat{\rho}(g, h)=N_{\rho}(g^{-1}h)$ for all $g, h\in F_{a}(X)$. We refer to $\hat{\rho}$ as the Graev extension of $\rho$ to $F_{a}(X)$.

Given a word $\mathscr{X}$ in the alphabet $\tilde{X}$, we say that $\mathscr{X}$ is {\it almost irreducible} if $\tilde{X}$ does not contain two consecutive symbols of the form $u$, $u^{-1}$ or $u^{-1}$, $u$ (but $\mathscr{X}$ may contain several letters equal to $e$), see \cite{RS}.

\bigskip

\section{Some generalized metric properties on free paratopological groups}
Let $X$ be a metrizable space and the topology be generated by a metric $d$. It follows from Theorem 3.2 in \cite{RS} that $d$ can be extended to an invariant metric $\hat{d}$ on $PG(X)$. Therefore, $PG(X)$ admits a weaker metrizable paratopological group topology.

For every non-negative integer $n$, denote by $B_{n}(X)$ the subspace of the free (Abelian) paratopological group $PG(X)$ that consists of all words of reduced length $\leq n$ with respect to the free basis $X$.

\begin{lemma}\cite[Proposition 7.6.1]{A2008}\label{l8}
Let $(X, d)$ be a metric space, and $\mathscr{F}_{d}$ be the topology on $F_{a}(X)$ generated by the Graev extension $\hat{d}$ of $d$ to $F(X)$. Then $B_{n}(X)$ is closed in $(F_{a}(X), \mathscr{F}_{d})$ for each $n\in \mathbb{N}$.
\end{lemma}

\begin{theorem}\label{t9}
Let $(X, d)$ be a metric space, and $\mathscr{F}_{d}$ be the topology on $FPX$ generated by the Graev extension $\hat{d}$ of $d$ to $FP(X)$. Then $B_{n}(X)$ is closed in $(F_{a}(X), \mathscr{F}_{d})$ for each $n\in \mathbb{N}$.
\end{theorem}

\begin{proof}
Since $X$ is Tychonoff, $FP(X)$ coincides algebraically with $F(X)$, and hence $\hat{d}$ on $FP(X)$ coincides with Graev's extension of $d$ over the
topological group $F(X)$. By Lemma~\ref{l8}, we complete the proof.
\end{proof}

By the similar proof of Theorem~\ref{t9}, we have the following theorem.

\begin{theorem}\label{t10}
Let $(X, d)$ be a metric space, and $\mathscr{F}_{d}$ be the topology on $A_{a}(X)$ generated by the Graev extension $\hat{d}$ of $d$ to $AP(X)$. Then $B_{n}(X)$ is closed in $(A_{a}(X), \mathscr{F}_{d})$ for each $n\in \mathbb{N}$.
\end{theorem}

\begin{proposition}\label{p0}
The sets $B_{n}(X)$ is closed in $FP(X)$.
\end{proposition}

\begin{proof}
Let $i: FP(X)\rightarrow F(X)$ be the identity
mapping. Then $i$ is a continuous isomorphism, and then $B_{n}(X)$ is closed in $FP(X)$ since each $B_{n}(X)$ is closed in $FP(X)$.
\end{proof}

By the similar proof of Proposition~\ref{p0}, we have the following
proposition.

\begin{proposition}\label{p1}
The sets $B_{n}(X)$ is closed in $AP(X)$.
\end{proposition}

By the definition of free paratopological groups, it is easy to obtain the following lemma.

\begin{lemma}\label{l11}\cite{RS}
The topology of the group $F_{a}(X)$ is the finest paratopological group topology on $F_{a}X$ that generates on $X$ its original topology. The same is valid for $AP(X)$.
\end{lemma}

Recalled that $X$ is a {\it $P$-space} if each $G_{\delta}$-set in $X$ is open.

\begin{theorem}\label{t8}
The groups $FP(X)$ and $AP(X)$ are $P$-spaces if and only if $X$ is a $P$-space.
\end{theorem}

\begin{proof}
Necessity is a consequence of the fact that subspaces of $P$-spaces are $P$-spaces.

Sufficiency. Let $X$ be a $P$-space. Suppose that $\mathscr{F}$ is the topology of $PG(X)$. Put $\mathscr{F}_{1}$ be the sets consisting of all $G_{\delta}$-sets in $PG(X)$. It is easy to see that $\mathscr{F}_{1}$ is a topology on $PG(X)$. Moreover, $(PG(X), \mathscr{F}_{1})$ is a paratopological group topology. Since $X$ is a $P$-space, the restrictions of both $\mathscr{F}$ and $\mathscr{F}_{1}$ to $X$ coincide the original topology of $X$. Therefore, ones have $\mathscr{F}=\mathscr{F}_{1}$ by Lemma~\ref{l11}. Therefore, $PG(X)$ is a $P$-space.
\end{proof}

We say that the topology of a space $X$ is {\it determined by a family} $\mathscr{C}$ of its subsets provided that a set $F\subset X$ is closed in $X$ iff $F\cap C$ is closed in $C$ for each $C\in\mathscr{C}$; We say that $G(X)$ ($PG(X)$) has the {\it direct limit property} if the topology of $G(X)$ ($PG(X)$) is determined by the family $\{B_{n}(X): n\in \mathbb{N}\}$.

\begin{theorem}
If $X$ is a $P$-space, then the groups $FP(X)$ and $AP(X)$ have the direct limit property.
\end{theorem}

\begin{proof}
Assume that $K$ is a subsets of $PG(X)$ such that $K\cap B_{n}(X)$ is closed in $B_{n}(X)$ for each $n\in \mathbb{N}$. It follows from Propositions~\ref{p0} and~\ref{p1} that the sets $B_{n}(X)$ are closed in $PG(X)$. Hence the sets $B_{n}(X)\cap K$ are closed in $PG(X)$. It follows that $K$ is an $F_{\sigma}$-set in $PG(X)$. By Theorem~\ref{t8}, we can see that $PG(X)$ is a $P$-space, which implies that $K$ is closed in $PG(X)$.
\end{proof}

By the group reflexion $G^{\flat}=(G, \tau^{\flat})$ of a paratopological group $(G, \tau)$ we understand the group $G$ endowed with the strongest topology $\tau^{\flat}\subset \tau$ such that $(G, \tau^{\flat})$ is a topological group.

A space is a {\it functionally Hausdorff space} if two distinct points $x$ and $y$ there is a continuous real-valued mapping $f$ on $X$ such that $f(x)\neq f(y)$.

In \cite{PN}, N.M. Pyrch and A.V. Ravsky proved that if $X$ is a Functionally Hausdorff space then the topological group $FP(X)^{\flat}$ and $F(X)$ are topological isomorphic. In fact, we can prove that if $X$ is a Functionally Hausdorff space then the topological group $AP(X)^{\flat}$ and $A(X)$ are topological isomorphic by an argument similar to the one in \cite{PN}. Therefore, we have the following lemma.

\begin{lemma}\label{l12}
If $X$ is a Functionally Hausdorff space then the topological group $PG(X)^{\flat}$ and $G(X)$ are topological isomorphic.
\end{lemma}

By Lemma~\ref{l12}, we can easily obtain the following lemma.

\begin{lemma}\label{l21}
If $X$ is a Functionally Hausdorff space and $f: X\rightarrow Y$ is the continuous embedding mapping, then $f$ admits an extension
to a continuous monomorphism $\hat{f}$: $PG(X)\rightarrow G(Y)$.
\end{lemma}

\begin{theorem}\label{t10}
Let $X$ be a Functionally Hausdorff space, and $A$ be an arbitrary subset of $PG(X)$. If $A\cap B_{n}(X)$ is finite for each $n\in\mathbb{N}$, then $A$ is closed and discrete in $PG(X)$.
\end{theorem}

\begin{proof}
Let $f: X\rightarrow Y$ be a topological embedding of $X$ to a compact space $Y$. It follows from Lemma~\ref{l21} that we can extend $f$ to a continuous monomorphism $\hat{f}: PG(X)\rightarrow G(Y)$. Let $B\subset A$, and put $C=\hat{f}(B)$. Then $C\cap B_{n}(Y)$ is finite for each $n\in \mathbb{N}$. Since $Y$ is compact, $G(Y)$ has the direct limit property \cite{GM}. Therefore, $C$ is closed in $G(Y)$. Since $\hat{f}$ is a continuous monomorphism, we have that $B$ is closed in $PG(X)$. Hence, all subsets of $A$ are closed in $PG(X)$, which implies that $A$ is discrete.
\end{proof}

\begin{theorem}
If $X$ is a Functionally Hausdorff space, and $K$ is a countably compact subspace of $PG(X)$, then $K\subset B_{n}(X)$ for some $n\in \mathbb{N}$.
\end{theorem}

\begin{proof}
Suppose that $K\setminus B_{n}(X)\neq\emptyset$ for each $n\in \mathbb{N}$. It is easy to see that we can find an infinite subset $A\subset K$ such that $A\cap B_{n}(X)$ is finite for each $n\in \mathbb{N}$. It follows from Theorem~\ref{t10}, Propositions~\ref{p0} and~\ref{p1} that $A$ is closed and discrete in $PG(X)$ and in $K$, which is a contradiction.
\end{proof}

\begin{theorem}
If $Y$ is a closed subspace of a Tychonoff space $X$, then the subgroup $PG(Y, X)$ of $PG(X)$ generated by $Y$ is closed in $PG(X)$.
\end{theorem}

\begin{proof}
Let $bX$ be a Hausdorff compactification of $X$, and $f: X\rightarrow bX$ be a topological embedding of $X$ to a compact space $bX$. It follows from Lemma~\ref{l21} that we can extend $f$ to a continuous monomorphism $\hat{f}: PG(X)\rightarrow G(bX)$. Denote by $Z$ the closure of $Y$ in $bX$. It follows from the compactness of $Z$ that $G(Z, bX)\cap B_{n}(bX)$ is compact for each $n\in \mathbb{N}$. Since $bX$ is compact, $G(bX)$ has the direct limit property \cite{GM}, and hence $G(Z, bX)$ is closed in $G(bX)$. Thus $G(Y, bX)=G(Z, bX)\cap G(X, bX)$ is a closed subgroup of $G(X, bX)$. Since $\hat{f}$ is a continuous monomorphism, ones have $\hat{f}(PG(Y, X))=G(Y, bX)$. Therefore, $PG(Y, X)$ is closed in $PG(X)$.
\end{proof}
\bigskip

\section{The character of free Abelian paratopological groups}
Firstly, we give some technical lemmas.

The following two lemmas are essentially claims in the proof of Theorem 3.2 in \cite{RS}.

\begin{lemma}\label{l0}\cite{RS}
Let $\varrho$ be a quasi-pseudometric on $X$ bounded by 1. If $g$ is a reduced word in $F_{a}(X)$ distinct from $e$, then there exists an almost irreducible word $\mathscr{X}_{g}=x_{1}x_{2}\cdots x_{2n}$ of length $2n\geq 2$ in the alphabet $\tilde{X}$ and a scheme $\varphi_{g}\in\mathscr{S}_{n}$ that satisfy the following conditions:
\begin{enumerate}
\item for $i=1, 2, \cdots, 2n$, either $x_{i}$ is $e$ or $x_{i}$ is a letter in $g$;

\item $[\mathscr{X}_{g}]=g$ and $n\leq \ell(g)$; and

\item $N_{\rho}(g)=\Gamma_{\rho}(\mathscr{X}_{g}, \varphi_{g}).$
\end{enumerate}
\end{lemma}

\begin{lemma}\label{l1}\cite{RS}
The family $\mathscr{N}=\{U_{\rho}(\varepsilon): \varepsilon >0\}$ is a base at the neutral element $e$ for a paratopological group topology $\mathscr{F}_{\rho}$ on $F_{a}(X)$, where $U_{\rho}(\varepsilon)=\{g\in F_{a}(X): N_{\rho}(g)<\varepsilon\}$. The restriction of $\mathscr{F}_{\rho}$ to $X$ coincides with the topology of the space $X$ generated by $\rho$.
\end{lemma}

\begin{lemma}\label{l2}\cite{F1982}
For every sequence $V_{0}, V_{1}, \cdots,$ of elements of a quasi-uniformity $\mathscr{U}$ on a set $X$, if $$V_{0}=X\times X\ \mbox{and}\ V_{i+1}\circ V_{i+1}\circ V_{i+1}\subset V_{i},\ \mbox{for}\ i\in \mathbb{N},$$ where `$\circ$' denotes the composition of entourages in the quasi-uniform space $(X, \mathscr{U})$, then there exists a quasi-pseudometric $\rho$ on the set $X$ such that, for each $i\in \mathbb{N}$, $$V_{i}\subset\{(x, y): \rho (x, y)\leq \frac{1}{2^{i}}\}\subset V_{i-1}.$$
\end{lemma}

\begin{lemma}\label{l3}\cite{LF}
For every quasi-uniformity $\mathscr{V}$ on a set $X$ and each $V\in \mathscr{V}$ there exists a quasi-pseudometric $\rho$ bounded by 1 on $X$ which is quasi-uniform with respect to $\mathscr{V}$ and satisfies the condition $$\{(x, y): \rho (x, y)< 1\}\subset V.$$
\end{lemma}

\begin{lemma}\label{l4}\cite{LF}
Suppose that $\rho$ is a quasi-pseudometric on a set $X$, and suppose that $m_{1}x_{1}+\cdots +m_{n}x_{n}$ is the normal form of an element $h\in F_{a}(X)\setminus\{e\}$ of the length $l=\sum_{i=1}^{n}|m_{i}|$. Then there is a representation
$$h=(-u_{1}+v_{1})+\cdots +(-u_{k}+v_{k}),$$
where $2k=l$ if $l$ is even and $2k=l+1$ if $l$ is odd, $u_{1}, v_{1}, \cdots , u_{k}, v_{k}\in\{\pm x_{1}, \cdots , \pm x_{n}\}$ (but $v_{k}=e$ if $l$ is odd), and such that
$$\hat{\rho}_{A}(e, h)=\sum_{i=1}^{k}\rho^{\ast}(u_{i}, v_{i}).$$

In addition, if $\hat{\rho}_{A}(e, h)<1$, then $l=2k$, and one can choose $y_{1}, z_{1}, \cdots , y_{k}, z_{k}\in\{x_{1}, \cdots , x_{n}\}$ such that
$h=(-y_{1}+z_{1})+\cdots +(-y_{k}+z_{k})$  and
$\hat{\rho}_{A}(e, h)=\sum_{i=1}^{k}\rho^{\ast}(y_{i}, z_{i})$.
\end{lemma}

\begin{theorem}\label{t0}\cite{LF}
Let $X$ be a Tychonoff space, and
let $\mathscr{P}_{X}$ be the family of all continuous quasi-pseudometrics from $(X\times X, \mathscr{U}_{X}^{-1}\times \mathscr{U}_{X})$ to $(\mathbb{R}, \mathscr{U}^{\star})$ which are bounded by 1. Then the sets $$V_{\rho}=\{g\in AP(X): \hat{\rho}_{A}(e, g)<1\}$$ with $\rho\in\mathscr{P}_{X}$ form a local base at the neutral element $e$ of $AP(X)$.
\end{theorem}

\begin{lemma}\label{l5}\cite{A2008}
Let $k\in\omega, p, k_{1},\cdots, k_{p}\in\mathbb{N}$ such that $\sum_{i=1}^{p}2^{-k_{i}}<2^{-k}$. Then we have
\begin{enumerate}
\item If $(X, \mathscr{U})$ is a quasi-uniform space and $\{U_{n}: n\in\omega\}$ a countable subcollection of $\mathscr{U}$ such that $U_{n+1}\circ U_{n+1}\circ U_{n+1}\subset U_{n}$ for each $n\in\omega$, then $U_{k_{1}}\circ\cdots\circ U_{k_{p}}\subset U_{k}$;

\item If $\{V_{i}: i\in\omega\}$ is a sequence of subsets of a group $G$ with the identity $e$ such that $e\in V_{i}$ and $V_{i+1}^{3}\subset V_{i}$ for each $i\in \omega$, then we have $V_{k_{1}}\cdots V_{k_{n}}\subset V_{r}$.
\end{enumerate}
\end{lemma}

\begin{lemma}\label{l6}
Let $m_{1}x_{1}+\cdots+m_{n}x_{n}$ be the normal form of an element $g\in A_{a}(X)\setminus\{e\}$ and let $d$ be a quasi-pseudometric on $X$. If $\sum_{i=1}^{n}m_{i}=0$, then there is an reduced representation of $g$ in the form $$g=(-z_{1}+t_{1})+\cdots(-z_{k}+t_{k})$$ such that $2k=\sum_{i=1}^{n}|m_{i}|$, $z_{j}, t_{j}\in\{x_{1}, \cdots, x_{n}\}$ for each $j\leq k$ and $\hat{d}_{A}(e, g)=\sum_{j=1}^{k}d(z_{j}, t_{j})$.
\end{lemma}

\begin{proof}
It follows from $\sum_{i=1}^{n}m_{i}=0$ that the number $m=\sum_{i=1}^{n}|m_{i}|$ has to be even. Let $m=2k$ for some $k\in \mathbb{N}$. By Lemma~\ref{l4}, the element $g$ has a reduced representation $\varphi$ of $g$ of the form $$g=(-u_{1}+v_{1})+\cdots+(-u_{k}+v_{k})$$ such that $$\hat{d}_{A}(e, g)=\Gamma (\varphi)=\sum_{j=1}^{k}d^{\ast}(u_{j}, v_{j}),$$where $u_{j}, v_{j}\in\{\pm x_{1}, \cdots, \pm x_{n}\}$ for each $j\leq k$. Obviously, each $-u_{j}+v_{j}$ has one of the following four forms: $a-b, -a+b, a+b, -a-b$ for some $a, b\in X$.

Suppose that $-u_{1}+v_{1}$ has the third form. Then we have $-u_{1}=a\in X$ and $v_{1}\in X$, and hence $-u_{1}+v_{1}=a+v_{1}$. Since $\sum_{i=1}^{n}m_{i}=0$, there exists a $2\leq j\leq k$ such that $-u_{j}+v_{j}$ has the fourth form. Without loss of generality, we may assume that $j=2$. Then $u_{2}\in X$ and $v_{2}=-b$ for some $b\in X$. Therefore, from our definition of the quasi-pseudometric $d^{\ast}$ on $X\cup\{e\}\cup (-X)$, it follows that the sum $\Gamma (\varphi)$ contains the part corresponding to $-u_{1}+v_{1}$ and $(-u_{2}+v_{2})$
\begin{eqnarray}
d^{\ast}(u_{1}, v_{1})+d^{\ast}(u_{2}, v_{2})&=&d^{\ast}(-a, v_{1})+d^{\ast}(u_{2}, -b) \nonumber\\
&=&d(e, a)+d(e, v_{1})+d(u_{2}, e)+d(b, e) \nonumber\\
&\geq&d^{\ast}(b, a)+d^{\ast}(u_{2}, v_{1})\nonumber
\end{eqnarray}
Replace the sum $(-u_{1}+v_{1})+(-u_{2}+v_{2})$ in the reduced representation $\varphi$ by $(-u_{2}+v_{1})+(-b+a)$. Therefore, we get another reduced representation $\varphi^{\prime}$ of $g$ of the form $$g=(-u_{2}+v_{1})+(-b+a)+(-u_{3}+v_{3})+\cdots+(-u_{k}+v_{k}).$$
It follows from the above inequality that $\Gamma (\varphi^{\prime})\leq\Gamma (\varphi)$. Since $k$ is finite, by induction, we can give another reduced representation $\varphi^{\prime\prime}$ of $g$ of the form $$g=(-z_{1}+t_{1})+\cdots(-z_{k}+t_{k}),$$where $z_{j}, t_{j}\in\{x_{1}, \cdots, x_{n}\}$ for each $j\leq k$. Moreover, it is easy to see that we have $\Gamma (\varphi^{\prime\prime})\leq\Gamma (\varphi)$. However, the definition of $\hat{d}_{A}(e, g)=\Gamma (\varphi)\leq \Gamma (\varphi^{\prime\prime})$ whence it follows that $\hat{d}_{A}(e, g)=\sum_{j=1}^{k}d(z_{j}, t_{j})$.
\end{proof}

Suppose that $\mathscr{U}_{X}$ is the finest quasi-uniformity of a space $X$. Put $^{\omega}\mathscr{U}_{X}=\{P: P\ \mbox{is a sequence of}\ \mathscr{U}_{X}\}$. For each $P\in\ ^{\omega}\mathscr{U}_{X}$, denote by $P=\{U_{1}, U_{2},\cdots\}$ or $P=\{U_{n}: n\in\omega\}$.

For each $P=\{U_{1}, U_{2},\cdots\}\in\ ^{\omega}\mathscr{U}_{X}$, let $$W(P)=\{-x_{1}+y_{1}-\cdots -x_{k}+y_{k}: (x_{i}, y_{i})\in U_{i}\ \mbox{for}\ i=1, 2, \cdots, k, k\in\mathbb{N}\},\ \mbox{and},$$ $\mathscr{W}=\{W(P): P\in\ ^{\omega}\mathscr{U}_{X}\}.$

Moreover, fixed any $n\in\mathbb{N}$. Let

$\mathscr{Q}_{n}(P)=\{Q\subset P: |Q|=n\};$

$W_{n}(P)=\{-x_{1}+y_{1}-\cdots -x_{n}+y_{n}: (x_{j}, y_{j})\in U_{i_{j}}\ \mbox{for}\ i=1, 2, \cdots, n,$ $\{U_{i_{1}}, U_{i_{2}}, \cdots, U_{i_{n}}\}\in \mathscr{Q}_{n}(P)\},$ and

$\mathscr{W}_{n}=\{W_{n}(P): P\in\ ^{\omega}\mathscr{U}_{X}\}$.

{\bf Remark} In the above definition, for $P\in\ ^{\omega}\mathscr{U}_{X}$, there may be the same elements in $P$. In particular, for every $U\in\mathscr{U}_{X}$, we have $\{U, U, \cdots\}$ is also in $^{\omega}\mathscr{U}_{X}$. Moreover, the reader should note that the representation of elements of $W(P)$ and $W_{n}(P)$ need not be a reduced representation.

Put $\mathscr{R}_{n}(P)=\{Q\subset P: |Q|\leq n\}$. It is easy to see that

$W_{n}(P)=\{-x_{1}+y_{1}-\cdots -x_{k}+y_{k}: (x_{j}, y_{j})\in U_{i_{j}}\ \mbox{for}\ i=1, 2, \cdots, k$, $\{U_{i_{1}}, U_{i_{2}}, \cdots, U_{i_{k}}\}\in \mathscr{R}_{n}(P)\}.$

The following theorem follows from an adaptation of the characterization of the
neighborhoods of the identity for free topological groups obtained by V.G. Pestov in \cite{PE}.

\begin{theorem}\label{t3}
The family $\mathscr{W}$ is a neighborhood base of $e$ in $AP(X)$.
\end{theorem}

\begin{proof}
It is easy to prove that the family $\mathscr{W}$ satisfies the following conditions (i)-(iv):\\
(i) for each $W\in\mathscr{W}$, there exists a $V\in\mathscr{W}$ such that $V+V\subset W$;\\
(ii) for each $W\in\mathscr{W}$ and each $g\in V$, there exists a $V\in\mathscr{W}$ such that $g+V\subset W$;\\
(iii) for every $U, V\in\mathscr{W}$, there exists a $W\in\mathscr{W}$ such that $W\subset U\cap V$;\\
(iv) $\{0\}=\cap\mathscr{W}$.

Therefore, the topology $\mathscr{F}_{1}$ generated by $\mathscr{W}$ on $A_{a}(X)$ is a paratopological group topology. Pick $P=\{U_{1}, U_{2},\cdots\}\in\ ^{\omega}\mathscr{U}_{X}$ and $x\in X$, and put $W(x)=\{y\in X: (x, y)\in U_{1}\}.$ Then $W(x)$ is open in $X$ since $\mathscr{U}_{X}$ is compactible with the orignal topology for $X$. Furthermore, we can prove that $x\in W(x)\subset (x+W(P))\cap X$, which implies that $\mathscr{F}_{1}|_{X}$ is weaker than the original topology for $X$.

Claim: The topology $\mathscr{F}_{1}$ is stronger than the topology of $AP(X)$.

Indeed, let $V$ be an open neighborhood of $e$ in $AP(X)$. Put $V_{0}=V$ and pick a sequence $\{V_{n}: n\in\omega\}$ of neighborhoods of $e$ in $AP(X)$ such that $V_{n}+V_{n}+V_{n}\subset V_{n-1}$. For each $n\in \mathbb{N}$, put $$U_{n}=\{(x, y)\in X\times X: -x+y\in V_{n}\},$$ then $U_{n}\in\mathscr{U}_{X}$. Hence $P=\{U_{1}, U_{2}, \cdots\}\in\ ^{\omega}\mathscr{U}_{X}.$ For each point $g\in W(P)$, then there exists an $n\in\mathbb{N}$ such that $g=-x_{1}+y_{1}-\cdots -x_{n}+y_{n}$ for some $(x_{i}, y_{i})\in U_{i}$ for $i=1, 2, \cdots, n.$ Therefore, it follows from Lemma~\ref{l5} that $g\in V_{1}+V_{2}+\cdots +V_{n}\subset V_{0}=V$. Then we have $W(P)\subset V.$

It follows from Claim that $\mathscr{F}_{1}|_{X}$ coincides with the original topology for $X$. Therefore, $\mathscr{F}_{1}$ is weaker than the topology for $AP(X)$. Hence $\mathscr{F}_{1}$ coincides with the topology for $AP(X)$. Thus the family $\mathscr{W}$ is a neighborhood base of $e$ in $AP(X)$.
\end{proof}

Using Lemma~\ref{l5} and Theorem~\ref{t3}, we can show the following theorem by a similar proof of \cite[Theorem 2.4]{YK}.

\begin{theorem}
For each $n\in \mathbb{N}$, the family $\mathscr{W}_{n}$ is a neighborhood base of $e$ in $A_{2n}(X)$.
\end{theorem}

\begin{theorem}\label{t4}
Let $d$ be a quasi-pseudometric on $X$. Then $\hat{d}_{A}(kx, ky)=kd(x, y)$ for all $x, y\in X$ and $k\in\mathbb{N}\cup\{0\}.$
\end{theorem}

\begin{proof}
If $x=y$ or $k=0$ then it is trivially true. Therefore, we can suppose that $x\neq y$ and $k\in \mathbb{N}$. Let $g=-kx+ky$.
It follows from Lemma~\ref{l6} that $g$ has an reduced representation of the form  $$g=(-z_{1}+t_{1})+\cdots(-z_{k}+t_{k}),$$ where $z_{j}, t_{j}\in\{x, y\}$ for each $j\leq k$ and $\hat{d}_{A}(e, g)=\sum_{j=1}^{k}d(z_{j}, t_{j})$. Since the above representation of $g$ is reduced and $k>0$, each $-z_{j}+t_{j}$ is equal to $-x+y$. Therefore, we have $$\hat{d}_{A}(kx, ky)=\hat{d}_{A}(e, g)=\sum_{j=1}^{k}d(z_{j}, t_{j})=\sum_{j=1}^{k}d(x, y)=kd(x, y).$$
\end{proof}

A pair $(P, \leq)$ is call a {\it quasi-ordered set} if $\leq$ is a reflexive transitive
relation on the set $P$. If $(P, \leq)$ has the additional property of antisymmetry,
then it is a {\it partially ordered set}. A set $D\subset P$ is called {\it dominating}
in the quasi-ordered set $(P, \leq)$ if for each $p\in P$ there exists $q\in D$ such that
$p\leq q$. Similarly, a subset $E$ of $P$ is said to be {\it dense} in $(P, \leq)$ if for every
$p\in P$ there exists $q\in E$ with $q\leq p$. The minimal cardinality of a dominating
family in $(P, \leq)$ is denoted by $D(P, \leq)$, while we use $d(P, \leq)$ for the minimal
cardinality of a dense set in $(P, \leq)$. The notions of dominating and dense sets
are dual: if a set $S$ is dense in $(P, \leq)$, then it is dominating in $(P, \geq)$ and
vice versa. Therefore, $d(P, \leq)=D(P, \geq)$ and $d(P, \geq)=D(P, \leq)$. Note that in any
homogeneous space $G$ we have $d(N(e), \subset)=\chi(G)$.

If $(P; \leq)$ and $(Q; \ll)$ are quasi-ordered sets, then a mapping $f: P\rightarrow Q$ is
called {\it order-preserving} if $x\leq y$ implies $f(x)\ll f(y)$, where $x, y\in P$. Similarly,
$f$ is order-reversing if $x\leq y$ implies $f(x)\gg f(y)$.

\begin{lemma}\cite{NP}\label{l7}
Let $(P, \leq)$ and $(Q, \ll)$ be quasi-ordered sets, and let $f: P\rightarrow Q$ be an order-preserving map. If $f(P)$ is a dominating set in $Q$, then $D(Q)\leq D(P).$
\end{lemma}

Let $X$ be a Tychonoff space, and let $\mathscr{P}_{X}$ be the family of all continuous quasi-pseudometrics from $(X\times X, \mathscr{U}_{X}^{-1}\times \mathscr{U}_{X})$ to $(\mathbb{R}, \mathscr{U}^{\star})$.

\begin{theorem}\label{t5}
Let $X$ be a Tychonoff space. Then $\chi (AP(X))=D(\mathscr{P}_{X}, \leq)$.
\end{theorem}

\begin{proof}
It follows from Theorem~\ref{t3} that there exists a natural correspondence between the family $\mathscr{P}_{X}$ and a base at the neutral element $e$ of $AP(X)$. Indeed, the map $d\mapsto V_{d}$ from $(\mathscr{P}_{X}, \leq)$ to the partially ordered set $(\mathscr{N}(e), \supseteq)$ of open neighborhoods of $e$ in $AP(X)$ is order-preserving and maps $(\mathscr{P}_{X}, \leq)$ to a base at $e$ in AP(X), that is, a dominating set in $(\mathscr{N}(e), \supseteq)$, which implies that $\chi (AP(X))\leq D(\mathscr{P}_{X}, \leq)$ by Lemma~\ref{l7}.

Assume that a subset $Q\subset \mathscr{P}_{X}$ is such that $\{V_{d}: d\in Q\}$ is a base at the neutral element in $AP(X)$.

Claim: For every $\rho\in\mathscr{P}_{X}$, there exists $d\in Q$ such that $\rho\leq 2d$.

Indeed, fixed a $\rho\in\mathscr{P}_{X}$, since $\{V_{d}: d\in Q\}$ is a base at the neutral element in $AP(X)$, there exists $d\in Q$ such that $V_{d}\subset V_{\rho}$. Now, we shall show that $\rho\leq 2d$. Let $x, y\in X$. If $d(x, y)<1$ then we have $-x+y\in V_{d}\subset V_{\rho}$, and hence $\rho(x, y)<1$. Given an $n\in \mathbb{N}$. Similarly, if $n\in\mathbb{N}$ and $d(x, y)< 2^{-n}$ then it follows from Theorem~\ref{t4} that $\hat{d}_{A}(e, 2^{n}(-x+y))=\hat{d}_{A}(2^{n}x, 2^{n}y)=2^{n}d(x, y)<1$. Therefore, we have $2^{n}(-x+y)\in V_{d}\subset V_{\rho}$, and hence $\hat{\rho}_{A}(2^{n}x, 2^{n}y)=2^{n}\rho(x, y)<1$, that is, $\rho(x, y)<2^{-n}$. Thus we have showed that $d(x, y)<2^{-n}$ implies $\rho(x, y)<2^{-n}$ for $n\in\mathbb{N}\cup\{0\}$. Hence, it is easy to see that $\rho(x, y)\leq 2d(x, y)$ if $d(x, y)=0$ or $d(x, y)=1$. If $0<d(x, y)<1$, then we can choose an $n\in \mathbb{N}$ such that $2^{-n-1}\leq d(x, y)<2^{-n}$. Therefore, we have $\rho(x, y)<2^{-n}$, which implies that $\rho(x, y)\leq 2d(x, y)$. Hence we have $\rho\leq 2d$.

For each $d\in Q$, let $d^{\ast}=\min\{2d, 1\}$, and put $Q^{\ast}=\{d^{\ast}: d\in Q\}$. Obviously, we have $Q^{\ast}\subset \mathscr{P}_{X}$. It follows from our claim that $Q^{\ast}$ is a dominating family in $\mathscr{P}_{X}$. Thus $D(\mathscr{P}_{X}, \leq)\leq |Q^{\ast}|\leq |Q|$.

Suppose that $\mathscr{B}$ is a base at $e$ in $AP(X)$. For each $B\in\mathscr{B}$, it follows from Theorem~\ref{t0} that there exists $d_{B}\in\mathscr{P}_{X}$ such that $V_{d_{B}}\subset B.$ Put $\mathscr{A}=\{d_{B}: B\in\mathscr{B}\}$. Therefore, the set $\{V_{d_{B}}: B\in\mathscr{B}\}$ is a base at $e$ and satisfies $|\mathscr{A}|\leq |\mathscr{B}|$, and hence $D(\mathscr{P}_{X}, \leq)\leq |\mathscr{A}|\leq |\mathscr{B}|$. Thus $D(\mathscr{P}_{X}, \leq)\leq\chi(AP(X))$.

Therefore, we have $D(\mathscr{P}_{X}, \leq)=\chi(AP(X))$.
\end{proof}

From the proof of Theorem~\ref{t5}, it is easy to show the following theorem.

\begin{theorem}
Let $X$ be a Tychonoff space and $\mathscr{Q}\subset\mathscr{P}_{X}$. Then the collection of open sets$$\{g\in AP(X): \hat{d}_{A}(e, g)<\varepsilon\},\ \mbox{for}\ d\in\mathscr{Q}\ \mbox{and}\ \varepsilon>0,$$is a base at $e$ for the topology of the free Abelian paratopological group $AP(X)$ if and only if for each $\rho\in\mathscr{P}_{X}$ there is $d\in\mathscr{Q}$ such that $\rho\leq 2d.$
\end{theorem}

Given two sequences $s=\{U_{n}: n\in\omega\}$ and $t=\{V_{n}: n\in\omega\}$ in $^{\omega}\mathscr{U}_{X}$, we write $s\leq t$ provided that $U_{n}\subset V_{n}$ for each $n\in\omega$.

\begin{theorem}\label{t6}
Let $X$ be a Tychonoff space. Then $\chi(AP(X))=d(^{\omega}\mathscr{U}_{X}, \leq)$.
\end{theorem}

\begin{proof}
It follows from Theorem~\ref{t3} that $\chi(AP(X))\leq d(^{\omega}\mathscr{U}_{X}, \leq)$. By Theorem~\ref{t5}, it is suffice to show that $d(^{\omega}\mathscr{U}_{X}, \leq)\leq D(\mathscr{P}_{X}, \leq)$.

Indeed, for each $d\in\mathscr{P}_{X}$ and $n\in\omega$, let
$$U_{n}(d)=\{(x, y)\in X\times X: d(x, y)\leq 2^{-n}\}.$$
Obvious, the correspondence $d\mapsto \{U_{n}(d): n\in\omega\}$ is an order-reversing mapping of $D(\mathscr{P}_{X}, \leq)$ to $(^{\omega}\mathscr{U}_{X}, \leq)$. Put $\mathscr{A}=\{\{U_{n}(d): n\in\omega\}: d\in\mathscr{P}_{X}\}$. Then $\mathscr{A}$ is a dense set in $^{\omega}\mathscr{U}_{X}$. In fact, take an arbitrary sequence $\{U_{n}: n\in\omega\}\in$ $^{\omega}\mathscr{U}_{X}$. Then there exists a sequence $\{V_{n}:n\in\omega\}$ such that $3V_{n+1}\subset V_{n}\subset U_{n}$ for each $n\in\omega$. It follows from Theorem~\ref{t0} that there exists $d\in\mathscr{P}_{X}$ such that $U_{n}(d)\subset V_{n}$ for each $n\in\omega$. Therefore, we have $\{V_{n}: n\in\omega\}\leq \{U_{n}: n\in\omega\}$.
\end{proof}

\begin{corollary}
Let $X$ be a Tychonoff space. Then $\omega(X, \mathscr{U}_{X})\leq \chi(AP(X))\leq \omega(X, \mathscr{U}_{X})^{\aleph_{0}}.$
\end{corollary}

\begin{proof}
It is clear that $$\omega(X, \mathscr{U}_{X})=d(\mathscr{U}_{X}, \subseteq)\leq d(^{\omega}\mathscr{U}_{X}, \leq)\leq d(\mathscr{U}_{X}, \subseteq)^{\aleph_{0}}=\omega(X, \mathscr{U}_{X})^{\aleph_{0}}.$$ Now, the result easily
follows from Theorem~\ref{t6}.
\end{proof}

We say that $X$ is a {\it quasi-uniform $P$-space} if the intersection of countably many elements of $\mathscr{U}_{X}$ is again an element of $\mathscr{U}_{X}$.

\begin{theorem}
If $X$ is a Tychonoff quasi-uniform $P$-space, then $\chi(AP(X))=\omega(X, \mathscr{U}_{X}).$
\end{theorem}

\begin{proof}
The correspondence $U\mapsto \{U, U, \cdots\}$ from $(\mathscr{U}_{X}, \subseteq)$ to $(^{\omega}\mathscr{U}_{X}, \leq)$ is an order-preserving embedding. Put $\mathscr{A}=\{\{U, U, \cdots\}: U\in\mathscr{U}_{X}\}$. Then $\mathscr{A}$ is a dense set in $^{\omega}\mathscr{U}_{X}$. Indeed, since $X$ is a quasi-uniform $P$-space, it follows that, for an arbitrary $\{U_{0}, U_{1}, \cdots\}\in$ $^{\omega}\mathscr{U}_{X}$, we have $U=\bigcap_{n\in\omega}U_{n}\in\mathscr{U}_{X}$, and hence $\{U, U, \cdots\}\leq \{U_{0}, U_{1}, \cdots\}$. Therefore, we have $\chi(AP(X))=\omega(X, \mathscr{U}_{X})$ by Theorem~\ref{t6}.
\end{proof}

Let $\omega^{\omega}$ denote the family of all functions from $\mathbb{N}$ into $\mathbb{N}$. For $f, g\in\omega^{\omega}$ we write $f<^{\ast}g$ if $f(n)<g(n)$ for all but finitely many $n\in \mathbb{N}$. A family $\mathscr{F}$ is {\it bounded} if there is a $g\in\omega^{\omega}$ such that $f<^{\ast}g$ for all $f\in\mathscr{F}$, and is {\it unbounded} otherwise. We denote by $\flat$ the smallest cardinality of an unbounded family in $\omega^{\omega}$. It is easy to see that $\omega <\flat\leq \mathrm{c}$, where $\mathrm{c}$ denotes the cardinality of the continuum.

By means of an argument
similar to the one used in \cite[Lemma 2.14]{NP} we obtain the following theorem.

\begin{theorem}\label{t7}
If a Tychonoff space $X$ is not a quasi-uniform $P$-space, then $\flat\leq\chi(AP(X)).$
\end{theorem}

\begin{theorem}
If a Tychonoff space $X$ contains an infinite compact set, then $\flat\leq\chi(AP(X)).$
\end{theorem}

\begin{proof}
Let $K$ be an infinite compact set of $X$. Since an infinite compact set cannot be a $P$-space, $X$ is not a quasi-uniform $P$-space. Then it follows from Theorem~\ref{t7} that we have $\flat\leq\chi(AP(X)).$
\end{proof}

{\bf Acknowledgements}. I wish to thank
the reviewers for the detailed list of corrections, suggestions to the paper, and all her/his efforts
in order to improve the paper.

\end{document}